\documentclass[a4paper,12pt,twoside,leqno,final]{amsart}
\usepackage{amsmath}
\usepackage{amssymb}

\newtheorem{thm}{Theorem}[section]

\newtheorem*{tha}{Theorem A}

\newcommand{\C}{{\mathbb C}}
\newcommand{\D}{{\mathbb D}}

\newcommand{\T}{{\mathbb T}}

\newcommand{\f}{\frac}

\newcommand{\al}{\alpha}

\newcommand{\la}{\lambda}
\newcommand{\ze}{\zeta}
\renewcommand{\th}{\theta}

\newcommand{\Om}{\Omega}

\numberwithin{equation}{section}

\title[On the singular factor of a linear combination]
{On the singular factor of a linear combination\\
of holomorphic functions\\
\quad\\
{\it Sur le facteur singulier d'une combinaison lin\'eaire\\
de fonctions holomorphes}}
\author{Konstantin M. Dyakonov}
\address{ICREA and Universitat de Barcelona, Departament de Matem\`atica 
Aplicada i An\`alisi, Gran Via 585, E-08007 Barcelona, Spain}
\email{konstantin.dyakonov@icrea.cat}
\thanks{Supported in part by grant MTM2011-27932-C02-01 from El Ministerio de Ciencia 
e Innovaci\'on (Spain) and grant 2009-SGR-1303 from AGAUR (Generalitat de Catalunya).}

\begin{document}
\begin{abstract}
We prove that the linear combinations of functions $f_0,\dots,f_n\in H^\infty$ have \lq\lq few" 
singular inner factors, provided that the $f_j$'s are suitably smooth up to the boundary, while 
in general this is no longer true. 

\bigskip

\noindent{\bf R\'esum\'e.} On d\'emontre que les combinaisons lin\'eaires des fonctions 
$f_0,\dots,f_n\in H^\infty$ poss\`edent \lq\lq peu" de facteurs singuliers, \`a condition 
que les $f_j$ soient suffisamment lisses jusqu'au bord, mais que ceci n'est pas vrai dans 
le cas g\'en\'eral. 
\end{abstract}

\maketitle

\section{Introduction}

The classical uniqueness theorem tells us that, given a non-null holomorphic function $f$ on a domain 
$\Om\subset\C$, the zero set of $f$ is discrete (i.\,e., has no accumulation points in $\Om$). 
This result admits an amusing extension to linear combinations of several functions. 

\begin{tha} Suppose $f_0,\dots,f_n$ are linearly independent holomorphic functions on a domain 
$\Om\subset\C$. Then there is a discrete subset $\mathcal E$ of $\Om$ with the following property: 
whenever $g$ is a nontrivial linear combination of $f_0,\dots,f_n$, the zeros of $g$ whose multiplicity 
exceeds $n$ are all contained in $\mathcal E$. 
\end{tha}

\par Thus, \lq\lq deep" zeros (i.e., those of multiplicity greater than $n$) are forbidden for a non-null 
linear combination $\sum_{j=0}^n\la_jf_j$ except on a \lq\lq thin" set, which depends only on the 
$f_j$'s but not on the $\la_j$'s. Needless to say, the assumption on the multiplicities can be 
neither dropped nor relaxed. Furthermore, the exceptional \lq\lq thin" sets that arise are the same as 
those in the classical uniqueness theorem, which corresponds to the case $n=0$. 

\par We strongly suspect that Theorem A must be known. However, having found no appropriate reference, 
we now give a quick proof. 
\par Let $W=W(f_0,\dots,f_n)$ denote the {\it Wronskian} of the functions $f_j$, so that 
\begin{equation}\label{eqn:wro}
W(f_0,\dots,f_n):=
\begin{vmatrix}
f_0&f_1&\dots&f_n\\
f'_0&f'_1&\dots&f'_n\\
\dots&\dots&\dots&\dots\\
f_0^{(n)}&f_1^{(n)}&\dots&f_n^{(n)}
\end{vmatrix}.
\end{equation}
The $f_j$'s being linearly independent {\it and holomorphic}, it follows that $W$ is non-null 
(and also holomorphic, of course), so its zeros form a discrete subset, say $\mathcal E$, of $\Om$. 
Now, a point $z\in\Om$ will be a zero of multiplicity at least $n+1$ for $\sum_{j=0}^n\la_jf_j$ 
if and only if 
\begin{equation}\label{eqn:sys}
\sum_{j=0}^n\la_jf_j^{(k)}(z)=0\qquad(k=0,\dots,n). 
\end{equation}
The condition for \eqref{eqn:sys} to have a nontrivial solution $(\la_0,\dots,\la_n)$ is that 
the Wronskian matrix $\{f_j^{(k)}(z)\}_{j,k=0}^n$ be singular, i.\,e., that $W(z)=0$. The points 
$z$ in question are, therefore, precisely those in $\mathcal E$. 

\par This proof actually shows that the \lq\lq deep" zeros of {\it all} the nontrivial linear 
combinations as above coincide with the zeros of a {\it single} holomorphic function, namely, of $W$. 

\par In what follows, the domain $\Om$ will be the disk $\D:=\{z\in\C:|z|<1\}$. 
We shall be concerned with functions in the {\it Hardy spaces} $H^p=H^p(\D)$ 
(see, e.\,g., \cite{G}) and in the {\it Hardy--Sobolev spaces} 
$$H^p_k:=\{f\in H^p:\,f^{(k)}\in H^p\},$$ 
where $p>0$ and $k$ is a positive integer. We remark that if the functions $f_0,\dots,f_n$ from Theorem A 
are taken to be in $H^p_n$, then their Wronskian $W$ is in $H^q$ with a suitable $q$ (in case $p\ge1$, 
this is true with $q=p$). The zero set $\mathcal E=W^{-1}(0)$ therefore satisfies the Blaschke 
condition $\sum_{z\in\mathcal E}(1-|z|)<\infty$, and we may rephrase Theorem A as saying that there 
exists a Blaschke product (the one built from $\mathcal E$) with a certain divisibility property. 

\par The purpose of this note is to prove an analog of the above result for {\it singular inner factors}. 
Recall that any nontrivial function $f\in H^p$ can be factored canonically as $f=FBS$, where $F$ is an 
outer function, $B$ a Blaschke product, and $S$ a singular inner function (see \cite[Chapter II]{G}). 
The last-mentioned factor is of the form 
$$S(z)=S_\mu(z):=\exp\left\{-\int_\T\f{\ze+z}{\ze-z}
d\mu(\ze)\right\},$$
where $\mu$ is a (positive) singular measure on the circle $\T:=\partial\D$. While the Blaschke product 
$B$ is determined by the zeros of $f$ in $\D$, the role of the singular factor $S$ is not so easily 
describable. In a sense, however, such factors can be thought of as responsible for the \lq\lq boundary 
zeros of infinite multiplicity". 

\par We prove then, in the spirit of Theorem A, that if $f_0,\dots,f_n$ satisfy the hypothesis of that 
theorem (with $\Om=\D$) and are suitably smooth up to $\T$, then there is a {\it single} singular inner 
function $S$ divisible by the singular factor of each nontrivial linear combination $\sum_{j=0}^n\la_jf_j$. 
This means that the totality of singular factors resulting from such linear combinations is rather poor. 
On the other hand, we give an example to the effect that the smoothness assumption is indispensable; 
in fact, it is not enough to assume that the $f_j$'s are merely in $H^\infty$. 

\section{Main result} 

\begin{thm}\label{thm:maintheorem} Let $f_0,\dots,f_n$ be linearly independent functions in $H^1_n$. 
Then there is a singular inner function $S$ with the following property: whenever $\la_0,\dots,\la_n$ 
are complex numbers with $\sum_{j=0}^n|\la_j|>0$, the singular factor of $\sum_{j=0}^n\la_jf_j$ divides $S$. 
\end{thm}

\begin{proof} Let 
\begin{equation}\label{eqn:ggg} 
g=\sum_{j=0}^n\la_jf_j
\end{equation}
be a nontrivial linear combination of the $f_j$'s, so that at least one of the coefficients (say, $\la_k$) is 
nonzero. Recalling the notation \eqref{eqn:wro} for the Wronskian, put 
$$W:=W(f_0,\dots,f_n)$$ 
and 
$$W_k:=W(f_0,\dots,f_{k-1},g,f_{k+1},\dots,f_n).$$ 
It should be noted that $W\not\equiv0$, because the $f_j$'s are linearly independent holomorphic functions, 
and also that $W\in H^1$. To verify the latter claim, expand the determinant \eqref{eqn:wro} along its last 
row and observe that the derivatives $f_j^{(k)}$ with $0\le k\le n-1$ are all in $H^\infty$. Furthermore, it 
follows from \eqref{eqn:ggg} that $W_k=\la_kW$. In particular, the inner factors of $W$ and $W_k$ are 
identical. 
\par Now, if $S_g$ is the singular factor of $g$, then the inner factors of $g',g'',\dots,g^{(n)}$ are all 
divisible by $S_g$. (Indeed, it is known that for every $h\in H^1_1$, the singular factor of $h$ divides that 
of $h'$; see \cite[Theorem 1]{Cau} or \cite[Lemma 2]{CS}.) Expanding the determinant $W_k$ along its $k$th 
column, $\left(g,g',\dots,g^{(n)}\right)^{T}$, and noting that the corresponding cofactors are in $H^1$, 
we conclude that $S_g$ divides the inner factor of $W_k$, or equivalently, of $W$. The function $S=S_W$, 
defined as the singular factor of $W$, has therefore the required property. 
\end{proof}

\section{An example} 

We borrow an idea from \cite{CS}. Let $\th$ be a nonconstant inner function that omits 
an uncountable set of values $E\subset\D$. (The existence of such a function with values 
in $\D\setminus E$, for any prescribed closed set $E$ of zero logarithmic capacity, 
is established in \cite[Chapter 2]{CL}.) For each $\al\in E$, one has 
\begin{equation}\label{eqn:tetal}
\th-\al=S_\al\cdot(1-\bar\al\th),
\end{equation}
with 
$$S_\al:=\f{\th-\al}{1-\bar\al\th}.$$ 
Here, $S_\al$ is a singular inner function (because $\al$ is not in the range of $\th$), while the other factor 
in \eqref{eqn:tetal} is outer. 
\par Write $\mu_\al$ for the singular measure associated with $S_\al$. For $\mu_\al$-almost every $\ze\in\T$, 
we have $S_\al(z)\to0$, and hence $\th(z)\to\al$, as $z\to\ze$ nontangentially; see \cite[Chapter II]{G}. It 
follows that the supports of $\mu_\al$'s, with $\al\in E$, are pairwise disjoint. The set $E$ being uncountable 
and the measures $\mu_\al$ nonzero, we readily deduce that no finite Borel measure $\mu$ on $\T$ can satisfy 
$\mu\ge\mu_\al$ for all $\al\in E$. Consequently, no singular inner function is divisible by every $S_\al$. 
\par Since $S_\al$ is the singular factor of $\th-\al$, which is a linear combination of $f_0=\th$ and $f_1=1$, 
we see that these $f_j$'s violate the conclusion of Theorem \ref{thm:maintheorem} with $n=1$. 

\section{Concluding remarks} 

(1) It would be interesting to know if the space $H^1_n$ in Theorem \ref{thm:maintheorem} can be replaced 
by a larger smoothness class (say, by $H^p_n$ with a $p<1$), and moreover, to find the optimal smoothness 
condition on the functions $f_j$. 

\smallskip

(2) The author's recent papers \cite{DCRM, DMA, DCONM} contain some related work, where smoothness properties 
of the Wronskian play a role when proving various extensions of the Mason--Stothers $abc$ theorem.

\end{document}